\documentclass[12pt]{article}
\usepackage{latexsym,amsfonts,amssymb,mathrsfs,float}

\usepackage{amsthm}

\usepackage{amsmath}

\usepackage[dvips]{color}
\usepackage{colordvi,multicol}

\newtheorem{thm}{Theorem}[section]

\newtheorem{lem}[thm]{Lemma}
\newtheorem{pro}[thm]{Proposition}

\newtheorem{rem}[thm]{Remark}

\numberwithin{equation}{section}

\begin{document}
\title{\bf Limiting behaviors for  longest consecutive switches in an IID Bernoulli sequence}
\author{Chen-Xu Hao and Ting Ma\thanks{Corresponding author: College of Mathematics, Sichuan University, Chengdu 610065, China\vskip 0cm E-mail address: 476924193@qq.com; matingting2008@yeah.net}\\ \\
 {\small College of Mathematics, Sichuan  University,  China}}

\maketitle
\date{}
\vskip 0.5cm \noindent{\bf Abstract}\quad
 In this paper we mainly discuss sharp lower and upper bounds for the length of  longest consecutive switches in IID Bernoulli sequences.
 This work is an extension of results in Erd\H{o}s and R\'{e}v\'{e}sz (1975) for longest head-run and  Hao et al.  (2021)  for longest consecutive switches in unbiased coin-tossing, and might be applied to reliability theory, biology, quality control, pattern recognition, finance, etc.

\smallskip


\smallskip


\section{Introduction}

A biased coin with two sides  is tossed independently and repeatedly, where 0 is used to denote ``tail" and 1 to denote ``head".
In the rest of this paper, we let $\{X_i,i\geq 1\}$ be a sequence of independent Bernoulli trails  with $P\{X_1=1\}=p$ and $P\{X_1=0\}=1-p=:q$ by assuming that $0<p<1$.

In unbiased case, Erd\H{o}s et al. (1970, 1975) studied the length of longest head-run's limit behaviors, these limit theorems have been extended in many subsequent studies. We refer to Schilling (1990), Binswanger and Embrechts (1994), Muselli (2000),  T\'{u}ri (2009), and Novak (2017). Mao et al. (2015) obtained a large deviation principle for  longest head run.  For more recent related references, we refer to  Liu and Yang (2016), Chen and Yu (2018), and Pawelec and Urba\'{n}ski (2020).

Anush (2012)  posed a question concerning the bounds for the number of coin tossing switches. According to Li (2013), a ``switch"  is a tail followed by a head or a head followed by a tail.  The exact probability distribution for the number of coin tossing switches was given by Joriki (2012). However, limiting behaviors of the related distribution cannot be obtained directly due to its complexity. Therefore, probabilistic estimates are necessary for applications such as in  molecular biology (Glaz and Naus (1991), Pincus and Singer (2016)), sensor networks (Paolo et al. (2015)),  and image detection (Ni et al. (2021)). Li (2013) considered the number of switches in unbiased coin-tossing, and established the central limit theorem and the large deviation principle for the total number of switches. In Hao et al. (2021), the authors obtained some limiting results for length of the longest consecutive switches in unbiased case with $p=q=1/2$.
The purpose of the present note is to extend the results in Hao et al. (2021) to IID Bernoulli sequences. As an application, it is reasonable to deduce similar results in the paper for the Markov case (see e.g.  in Samarova (1981), Vaggelatou  (2003), Liu et al. (2018), and Mezhennaya (2019)).

In the begining, we introduce some notations. Following Li (2013), for $m,n\in\mathbb N$ define
\begin{equation*}\label{def-SNM}
S_n^{(m)}:=\sum_{i=m+1}^{n+m-1}(1-X_{i-1})X_i+X_{i-1}(1-X_i)
\end{equation*}
as the total number of  switches  in $n$ trails $\{X_m,X_{m+1}, \ldots,X_{m+n-1}\}$.
 Moreover, for $m,N\in\mathbb N$ and $n=1,\ldots, N$, define by
$$
H_{m,n}^{(N)}:=\bigcup_{i=m}^{ m+N-n+1}\{S_n^{(i)}=n-1\}
$$
 the set of consecutive switches of length $n-1$ in $N$ trails $\{X_{m},X_{m+1},\ldots,X_{m+N-1}\}$. Then $M_N^{(m)}$, defined as
\begin{equation}\label{NUM-SWITCH}
M_N^{(m)}:=\max_{1\leq n\leq N}\left\{n-1|~H_{m,n}^{(N)}\neq \emptyset\right\},
\end{equation}
stands for {\it the length of longest consecutive switches} in $N$ trails  $\{X_{m},X_{m+1}, \ldots,X_{m+N-1}\}$. For example, if one gets ``$11001011101$", then it includes 6 switches and the length of longest consecutive switches is 3.
When $m=1$,  the length of longest consecutive switches in $N$ trails  $\{X_{1}, \ldots,X_{N}\}$ is  particularly denoted by $M_N$.

Throughout the paper, we use $\left\lfloor a\right\rfloor$ to denote the largest integer which  is no more than any real number $a$.

The rest of this paper is organized as follows.  In Section 2 we present main results. The proofs are given in Section 3.

\section{Main results }\label{MAIN-RES}
As discussed in Erd\H{o}s et al. (1975) and Hao et al. (2021), we mainly investigate sharp bounds  for  longest  consecutive switches. The main idea is similar. However, it  is much more complicated when the coin-tossing is biased. We have to discuss far more different situations to give precise calculations and probabilistic estimates for  longest consecutive switches  (see Lemma \ref{LEM-S2N} and Theorem \ref{THM-EST-SN1}).

In the first two theorems we give lower and upper bounds for   longest  consecutive switches. Furthermore, we will explain in Remark \ref{sharp} that they are sharp.
\begin{thm}\label{THM01}
Let $\varepsilon$ be any positive number. Then for almost all $\omega\in \Omega$, there exists a finite $N_0=N_0(\omega,\varepsilon)$ such that for $N\geq N_0$
\begin{equation} \label{alpha2}\begin{aligned}
M_N\geq &\left\lfloor \log_{1/\sqrt{pq}} N-\log_{1/\sqrt{pq}}\log_{1/\sqrt{pq}}\log_{1/\sqrt{pq}} N+\log_{1/\sqrt{pq}}\log_{1/\sqrt{pq}} e-\log_{1/\sqrt{pq}}2- 1- \varepsilon\right\rfloor\\
=:&\alpha_1(N).
\end{aligned}
\end{equation}
\end{thm}

\begin{thm}\label{THM02}
For almost all $\omega\in \Omega$, there exists an infinite sequence $N_i=N_i(\omega) (i=1,2,...)$ of integers such that
\begin{equation}\label{alpha02}
M_{N_i}<\left\lfloor \log_{1/\sqrt{pq}} N_i-\log_{1/\sqrt{pq}}\log_{1/\sqrt{pq}}\log_{1/\sqrt{pq}} N_i+\log_{1/\sqrt{pq}}\log_{1/\sqrt{pq}} e+1\right\rfloor=:\alpha_2(N_i).
\end{equation}
\end{thm}

According to Theorems \ref{THM01} and \ref{THM02}, we know that the value of $M_N$ is larger than $\alpha_1$ but in general not larger than $\alpha_2$. In the next  theorem we discuss the largest possible values of $M_N$.

\begin{thm}\label{THM03}
Let $\{\gamma_n\}$ be a sequence of positive numbers.

(i) If  $\sum_{n=1}^{\infty}{(pq)}^{\frac{\gamma_n}{2}}=\infty$, then  for almost all $\omega\in \Omega$, there exists an infinite sequence $N_i=N_i(\omega,\{\gamma_n\}) \ (i=1,2,...)$ of integers such that
$M_{N_i}\geq \gamma_{N_i}-1$.

(ii) If $\sum_{n=1}^{\infty}{(pq)}^{\frac{\gamma_n}{2}}<\infty$, then for almost all $\omega\in \Omega$, there exists a positive integer $N_0=N_0(\omega,\{\gamma_n\})$ such that $M_N< \gamma_N -1$
for all  $N\geq N_0$.
\end{thm}

The above theorem can be reformulated  to estimate number of  switches as follows:

\noindent {\bf Theorem 2.3*} {\it \label{THM03*}
Let $\{\gamma_n\}$ be a sequence of positive numbers.

(i) If $\sum_{n=1}^{\infty}{(pq)}^{\frac{\gamma_n}{2}}=\infty$,  then for almost all $\omega\in \Omega$, there exists an infinite sequence $N_i=N_i(\omega,\{\gamma_n\}) \ (i=1,2,...)$ of integers such that
$S^{(N_i-\gamma_{N_i})}_{\gamma_{N_i}}\geq \gamma_{N_i}-1$.

(ii) If $\sum_{n=1}^{\infty}{(pq)}^{\frac{\gamma_n}{2}}<\infty$, then for almost all $\omega\in \Omega$, there exists a positive integer $N_0=N_0(\omega,\{\gamma_n\})$ such that $S^{(N-\gamma_N)}_{\gamma_N}<\gamma_N -1$
for all $N\geq N_0$.}


In the end, we give a limit result on the length of the longest consecutive switches.

\begin{pro} \label{THM-MN}
We have
\begin{equation} \label{alpha1}
\lim_{N\to\infty}\frac{M_N}{\log_{1/\sqrt{pq}} N}=1 \quad a.s.
\end{equation}
\end{pro}

\section{Proofs}

\subsection{Proofs for Theorems \ref{THM01} - 2.3*}
The basic idea comes from  Erd\H{o}s and  R\'{e}v\'{e}sz (1975). At first, we give an estimate for the length of longest consecutive switches, which is very useful in our proofs for Theorems \ref{THM01} - 2.3*.

\begin{thm} \label{THM-EST-SN1} Let $N,K\in\mathbb N$ with $N\geq 2K$. Then
\begin{equation}\label{THM121}
\begin{aligned}
\Big(1-(K+1-2Kpq)&(pq)^{\frac{K-1}{2}}+(1-2pq)(pq)^{K-1}\Big)^{
\left\lfloor\frac{N}{K}\right\rfloor-1}    \leq P\Big(M_N<K-1\Big)\\
&\leq\Big(1-(K+2)(pq)^{\frac{K}{2}}+2(pq)^K\Big)^{\frac{1}{2}
\left(\left\lfloor\frac{N}{K}\right\rfloor-1\right)}.
\end{aligned}
\end{equation}
\end{thm}

To prove Theorem \ref{THM-EST-SN1}, we need the following lemma.

\begin{lem}\label{LEM-S2N}
Let $K,m\in\mathbb N$ and $M_{2K}^{(m)}$ is defined in (\ref{NUM-SWITCH}). Then
\begin{equation} \label{Lem.32}
P\left(M_{2K}^{(m)}\geq K-1\right)=\left\{
\begin{aligned}
&(K+2)(pq)^{\frac{K}{2}}-2(pq)^K, &\text{if~$K$~is~even},\\
&(K+1-2Kpq)(pq)^{\frac{K-1}{2}}-(1-2pq)(pq)^{K-1}, &\text{if~$K$~is~odd}.
\end{aligned}\right.
\end{equation}
\end{lem}

\begin{proof} Since $M_{2K}^{(m)},m\in\mathbb N$ are identically distributed, it is sufficient to consider the case where $m=1$.
For $i=1,\ldots,K+1$, denote
\begin{equation*}
\begin{aligned}
F_{i}:=&\{(X_{i},\cdots,X_{K+i-1})\ \mbox{is  the first  section of consecutive switches of length}\ K-1\\
&\quad\quad\quad\quad\quad\quad\quad\quad\   \mbox{in the sequence}\ (X_{1},\cdots,X_{2K})\}.
\end{aligned}
\end{equation*}
It is equivalent to
\begin{equation*}
\begin{aligned}
F_{1}=&\{M_K= K-1\},\\
F_{i}=&\{M_K^{(i)}= K-1,
 X_{i-1}=X_{i}\},~~i=2,\ldots,K,\\
F_{K+1}=&\{M_K^{(K+1)}= K-1,
 X_{i-1}=X_{i},M_K< K-1\}.
\end{aligned}
\end{equation*}
Then we have the decomposition
\begin{equation}\label{MFi}
\{M_{2K}\geq K-1\}=\bigcup_{i=1}^{K+1}F_{i},
\end{equation}
where $F_{i}\cap F_{j}=\emptyset,~\forall i\ne j$, and
\begin{equation}\label{F1p}
P(F_{1}) =\left\{
\begin{aligned}
&2(pq)^{\left\lfloor{K}/{2}\right\rfloor}, &\text{if~$K$~is~even},\\
&(pq)^{\left\lfloor {K}/{2}\right\rfloor}, &\text{if~$K$~is~odd},
\end{aligned}\right.~~~~~~~~~~~~~~~~~~~~~~~~~~~~~
\end{equation}

\begin{equation}\label{Fip}
P(F_{i}) =\left\{
\begin{aligned}
&(pq)^{\left\lfloor{K}/{2}\right\rfloor}, &\text{if~$K$~is~even},\\
&(p^2+q^2)(pq)^{\left\lfloor{K}/{2}\right\rfloor}, &\text{if~$K$~is~odd},
\end{aligned}\right.~~i=2,\ldots,K,
\end{equation}

\begin{equation}\label{Fk+1p}
P(F_{K+1}) =\left\{
\begin{aligned}
&(pq)^{\left\lfloor{K}/{2}\right\rfloor}\big(1-2(pq)^{\left\lfloor{K}/{2}\right\rfloor}\big), &\text{if~$K$~is~even},\\
&(pq)^{\left\lfloor{K}/{2}\right\rfloor}(p^2+q^2)\big(1-(pq)^{\left\lfloor{K}/{2}\right\rfloor}\big), &\text{if~$K$~is~odd}.
\end{aligned}\right.~~~~
\end{equation}
Plugging the above equalities into \eqref{MFi} and by $p+q=1$ we  obtain  \eqref{Lem.32}.
\end{proof}

\begin{proof}[{\bf Proof of Theorem \ref{THM-EST-SN1}}]
Let $N,K\in\mathbb N$ with $N\geq 2K$. Denote
\[
  B_j=\left\{M_K^{(j+1)}=K-1\right\},~j=0,1,\ldots,N-K.
\]
Then we have the decomposition
\begin{eqnarray*}
&&C_l:= \left\{M_{2K}^{(lK+1)}\geq K-1\right\}=\bigcup^{(l+1)K}_{j=lK}B_j,~l=0,1,\ldots,\lfloor(N-2K)/{K}\rfloor,
\end{eqnarray*}
\begin{equation*}
\left\{ M_N\geq K-1\right\}=\bigcup^{\left\lfloor\frac{N-2K}{K}\right\rfloor}_{l=0}C_l.
\end{equation*}
Let
\begin{equation}\label{D_01}
\begin{aligned}
D_0:=&C_0\cup C_2\cup\cdots\cup C_{2\left\lfloor\frac{1}{2}\left\lfloor\frac{N-2K}{K}\right\rfloor\right\rfloor},~~
D_1:=C_1\cup C_3\cup\cdots\cup C_{2\left\lfloor\frac{1}{2}(\left\lfloor\frac{N-2K}{K}\right\rfloor-1)\right\rfloor+1}.
\end{aligned}
\end{equation}
By the independence of events $C_0,C_2,\ldots,C_{2\left\lfloor\frac{1}{2}\left\lfloor\frac{N-2K}{K}\right\rfloor\right\rfloor}$, Lemma \ref{LEM-S2N}, and the independence of events $C_1,C_3,\ldots,C_{\left\lfloor\frac{1}{2}\left\lfloor\frac{N-2K}{K}\right\rfloor\right\rfloor+1}$, we have
\begin{equation*}\label{D0}
P(\overline{D_0})={P(\overline {C_0})}^{\left\lfloor\frac{1}{2}\left\lfloor\frac{N-2K}{K}\right\rfloor\right\rfloor+1},~~
P(\overline{D_1})=   P(\overline {C_0})^{\left\lfloor\frac{1}{2}\left(\left\lfloor\frac{N-2K}{K}\right\rfloor-1\right)\right\rfloor+1 }.
\end{equation*}
Observing that $\{ M_N\geq K-1\}=D_0\cup D_1$, it is obvious that
\begin{equation}\label{rest}
P(M_N<K-1)\leq \min \left\{P(\overline {D_0}),P(\overline {D_1})\right\}\leq P(\overline {C_0})^{(\lfloor N/K\rfloor-1)/2}.
\end{equation}
Moreover, it can be obtained that  for any   $j=0,1,\ldots,N-K$,
\[
P(D_1\mid B_{j})\geq P(D_1).
\]
 Similarly,  for each even $l_0=2,4,\cdots,2\left\lfloor\frac{1}{2}\left\lfloor\frac{N-2K}{K}\right\rfloor\right\rfloor$,
\[
P\big( D_1\bigcup_{{l~ \text{even},~0\leq l\leq l_0-2}}C_l\mid C_{l_0}\big)\geq P\big(D_1\bigcup^{}_{{l~ \text{even},~0\leq l\leq l_0-2}}C_l\big),
\]
equivalently,
\[
P\big(\overline{ D_1}\bigcap_{{l~ \text{even},~0\leq l\leq l_0-2}}\overline{C_l}\mid\overline{C_{l_0}}\big)\geq P\big(\overline{D_1}\bigcap_{{l~ \text{even},~0\leq l\leq l_0-2}} \overline{C_l}\big),
\]
which implies that
\begin{equation}\label{cont-p}
P\big(\overline{ D_1}\bigcap_{{l~ \text{even},~0\leq l\leq l_0}}\overline{C_l}\big)\geq P\big(\overline{D_1}\bigcap_{{l~ \text{even},~0\leq l\leq l_0-2}} \overline{C_l}\big) P(\overline{C_{l_0}}).
\end{equation}
Based on \eqref{cont-p} and the independence of events $C_0,C_2,\ldots,C_{2\left\lfloor\frac{1}{2}\left\lfloor\frac{N-2K}{K}\right\rfloor\right\rfloor}$
\begin{eqnarray}\label{lest}
P(M_N<K-1)& = &P\big(\overline{D_1}~\overline{D_0}\big) =P\big(\overline{D_1} \bigcap_{{l~ \text{even},~0\leq l\leq \left\lfloor\frac{N-2K}{K}\right\rfloor}} \overline{C_l} \big)\nonumber\\
 & \geq & P\big(\overline{D_1} \bigcap_{{l~ \text{even},~0\leq l\leq \left\lfloor\frac{N-2K}{K}\right\rfloor-2}} \overline{C_l} \big )P\big(\overline{C_{2\left\lfloor\frac{1}{2}\left\lfloor\frac{N-2K}{K}\right\rfloor\right\rfloor}}\big)\nonumber\\
 & \geq &\cdots \nonumber\\
 &\geq & P(\overline{D_1})\prod_{ {l~ \text{even},~0\leq l\leq \left\lfloor\frac{N-2K}{K}\right\rfloor} }  P(\overline{C_l})\nonumber\\
 & = &P(\overline{D_1})P(\overline{D_0})\nonumber\\
 & \geq &P(\overline{C_0})^{\lfloor N/K\rfloor-1}.
\end{eqnarray}
 Besides, by $pq\leq 1/4$ we consider in \eqref{Lem.32}  for any integer $K\geq 2$
\[
 (K+2)(pq)^{\frac{K}{2}}-2(pq)^K \leq
(K+1-2Kpq)(pq)^{\frac{K-1}{2}}-(1-2pq)(pq)^{K-1}.
\]
Plugging \eqref{Lem.32} into \eqref{rest} and \eqref{lest}, we complete the proof.
\end{proof}

Similarly as in Hao et al. (2021), we present the following result to prove Theorem 2.2.
\begin{lem}\label{Lem01}
Let $\{\alpha_j,~j\geq 1\}$ be a sequence of positive numbers. Suppose that $\lim\limits_{j\to \infty}\alpha_j=a>0$. Then we have
\begin{equation}\label{conv1}
\sum_{j=1}^{\infty}\alpha_j^{-\log_{1/\sqrt{pq}} j}<\infty~~\text{if $a>\frac{1}{\sqrt{pq}}=\lambda$},
\end{equation}
and
\begin{equation}\label{disconv1}
\sum_{j=1}^{\infty}\alpha_j^{-\log_{1/\sqrt{pq}} j}=\infty~~\text{if $a<\frac{1}{\sqrt{pq}}=\lambda$}.
\end{equation}
\end{lem}

\begin{rem}  \label{cong-wor}
It is well known  that when  $\lim\limits_{j\to \infty}\alpha_j=\lambda$, one cannot judge if $\sum_{j=1}^{\infty}\alpha_j^{-\log_{1/\sqrt{pq}}j}$ converges or not. For instance, if we take $\alpha_j=\big(j(\log_{1/\sqrt{pq}}j)^p\big)^{\frac{1}{\log_{1/\sqrt{pq}}j}}$  satisfying  $\lim_{j\to\infty}\alpha_j=\lambda $ for $p>0$, then  $\sum_{j=1}^{\infty}\alpha_j^{-\log_{1/\sqrt{pq}} j}$ converges if $p>1$ and diverges if $0<p\leq1$.
\end{rem}

\begin{proof} [{\bf Proof of Theorem \ref{THM01}}]
Let $N_j$ be the smallest integer with $\alpha_1(N_j)=j-1$. Then
\begin{eqnarray}\label{Theorem-st1}
\sum_{j=1}^{\infty}P\left(M_{N_j}<\alpha_1(N_j)\right)
&\leq & \sum_{j=1}^{\infty} \left(1-(j+2)(pq)^{{j}/{2}}+2(pq)^j\right)^{(\lfloor{N_j}/{j}\rfloor-1)/2}\nonumber\\
&\lesssim & \sum_{j=1}^{\infty}\left(1-(j+2)(pq)^{{j}/{2}}+2(pq)^j\right)^{{N_j}/{(2j)}}\nonumber\\
&=: & \sum_{j=1}^{\infty} e_j^{-{\{(j+2)(pq)^{{j}/{2}}-2(pq)^j\}}N_j/{(2j)}},
\end{eqnarray}
with  $e_j:=\big (1-(j+2)(pq)^{{j}/{2}}+2(pq)^j\big)^{\{2(pq)^j-(j+2)(pq)^{{j}/{2}}\}^{-1}}$.  The denotation $x\lesssim y $ means that there exists a constant $C$ independent of all variables  such that $x\leq Cy$.  The first inequality in \eqref{Theorem-st1} follows by the  inequality on the right hand side of  \eqref{THM121}. The second inequality in (\ref{Theorem-st1}) holds since by $pq\leq \frac 14$, when $j\to\infty$ we have $0\leq(j+2)(pq)^{{j}/{2}}-2(pq)^j\to 0$, and it implies that $ 1- (j+2)(pq)^{{j}/{2}}+2(pq)^j\in[1/2,1]$ for large $j$.
Moreover, by  $\alpha_1(N_j)=j-1$  we have
\begin{equation*}
 j\leq \log_{1/\sqrt{pq}} N_j-\log_{1/\sqrt{pq}}\log_{1/\sqrt{pq}}\log_{1/\sqrt{pq}} N_j+\log_{1/\sqrt{pq}}\log_{1/\sqrt{pq}} e-\log_{1/\sqrt{pq}} 2-\varepsilon,
 \end{equation*}
from which for large $j$
\begin{eqnarray*}
{N_j}\cdot{(pq)^{\frac{j}{2}}}\geq 2 \lambda^{\varepsilon}\cdot \ln \lambda \cdot\log_{1/\sqrt{pq}}{\log_{1/\sqrt{pq}}{N_j}}\geq 2 \lambda^{\varepsilon} \cdot\ln \lambda \cdot\log_{1/\sqrt{pq}}j.
\end{eqnarray*}
Then in \eqref{Theorem-st1} we get
\begin{eqnarray}\label{reason01}
 \sum_{j=1}^{\infty}P\left(M_{N_j}<\alpha_1(N_j)\right)
\leq  C\sum_{j=1}^{\infty} {e_j}^{- \left\{\left(\lambda^\epsilon  \ln \lambda \cdot {(j+2-2(pq)^{{j}/{2}})} /{j}\right)\cdot\log_{1/\sqrt{pq}} j\right\} },
\end{eqnarray}
where $\lim\limits_{j\to\infty}e_j=e$.  Again by $pq\leq1/4$ for  any $\varepsilon>0$,
\begin{eqnarray}\label{reason02}
 \lim_{j\to\infty} {e_j}^{\lambda^\epsilon  \ln \lambda \cdot{(j+2-2(pq)^{{j}/{2}})}/ {j}}
 = \lambda^{\lambda^\varepsilon}> \lambda.
\end{eqnarray}
Then by Lemma $\ref{Lem01}$ and the Borel-Cantelli lemma, we complete the proof.
\end{proof}

To prove Theorem $\ref{THM02}$, we need the following version of Borel-Cantelli lemma.

\begin{lem}{\rm(Erd\H{o}s and R\'{e}v\'{e}sz (1975), Lemma A)}\label{B_C}
Let $A_1,A_2,\ldots$ be arbitrary events, satisfying the conditions
$
\sum_{n=1}^{\infty}P(A_n)=\infty
$
and
\begin{eqnarray}\label{condition2}
\liminf_{n\to \infty}\frac{\sum_{1\leq k<l\leq n}P(A_kA_l)}{\sum_{1\leq k<l\leq n}P(A_k)P(A_l)}=1.
\end{eqnarray}
Then  $P\left(\limsup\limits_{n\to\infty}A_n\right)=1$.
\end{lem}

\begin{proof}[{\bf Proof of Theorem \ref{THM02}}]
For $\delta>0$, let $N_j=N_j(\delta)$ be the smallest integer for which $\alpha_2(N_j)=\lfloor j^{1+\delta}\rfloor-1$ with $\alpha_2(N_j)$ given by $(\ref{alpha02})$.
Set
\begin{equation}\label{A_j}
A_j=\left\{M_{N_j}<\alpha_2(N_j)\right\},~~j\geq1.
\end{equation}
 By Theorem \ref{THM-EST-SN1}, we have
\begin{eqnarray}\label{proof-thm3.3-1}
\sum_{j=1}^{\infty}P(A_j)
&\geq&\sum_{j=1}^{\infty}\Big(1-(\lfloor j^{1+\delta}\rfloor+1-2\lfloor j^{1+\delta}\rfloor pq)(pq)^{\frac{\lfloor j^{1+\delta}\rfloor -1}{2}}+(1-2pq)(pq)^{\lfloor j^{1+\delta}\rfloor-1}   \Big)^{\big\lfloor\frac{N_j}{\lfloor j^{1+\delta}\rfloor}\big\rfloor-1}\nonumber\\
&\geq&\sum_{j=1}^{\infty}\Big(1-(\lfloor j^{1+\delta}\rfloor +1-2\lfloor j^{1+\delta}\rfloor pq)(pq)^{\frac{\lfloor j^{1+\delta}\rfloor -1}{2}} \Big)^{\frac{N_j}{\lfloor j^{1+\delta}\rfloor}}\nonumber\\
&=:&\sum_{j=1}^{\infty}f_j^{-V_j\cdot \frac{N_j}{\lfloor j^{1+\delta}\rfloor}},
\end{eqnarray}
where $f_j=(1-V_j )^{-{V_j}^{-1}}$ with
$V_j:=(\lfloor j^{1+\delta}\rfloor +1-2\lfloor j^{1+\delta}\rfloor pq)(pq)^{{(\lfloor j^{1+\delta}\rfloor-1)}/{2}}$.
By  $\alpha_2(N_j)=\lfloor j^{1+\delta}\rfloor -1$ we have
\[
\log_{1/\sqrt{pq}} N_j-\log_{1/\sqrt{pq}}\log_{1/\sqrt{pq}}\log_{1/\sqrt{pq}} N_j+\log_{1/\sqrt{pq}}\log_{1/\sqrt{pq}} e+1\leq\lfloor j^{1+\delta}\rfloor,
\]
which implies that
\begin{equation}\label{es-1}
N_j\cdot (pq)^{\frac{\lfloor j^{1+\delta}\rfloor -1}{2}}\leq {\ln{\lambda} \cdot\log_{1/\sqrt{pq}}\log_{1/\sqrt{pq}} N_j} .
\end{equation}
Define $0<\varepsilon_0<1$ satisfying for any $j\geq 1$,
$$
\lfloor j^{1+\delta}\rfloor=\alpha_2(N_j)+1>\varepsilon_0\log_{1/\sqrt{pq}} N_j.
$$
Then we  have
\begin{equation}\label{proof-thm3.3-5}
\log_{1/\sqrt{pq}}\log_{1/\sqrt{pq}} N_j\leq (1+\delta)\log_{1/\sqrt{pq}} j-\log_{1/\sqrt{pq}}\varepsilon_0.
\end{equation}
Plugging \eqref{es-1} and \eqref{proof-thm3.3-5} into (\ref{proof-thm3.3-1}) we get
\begin{equation}\label{proof-thm3.3-6}
\sum_{j=1}^{\infty}P(A_j)
\geq\sum_{j=1}^{\infty}f_j^{-\big\{\ln{\lambda}\cdot{\frac{\lfloor j^{1+\delta}\rfloor+1-2\lfloor j^{1+\delta}\rfloor pq}{\lfloor j^{1+\delta}\rfloor}
\cdot\frac{(1+\delta)\log_{1/\sqrt{pq}} j-\log_{1/\sqrt{pq}}{\varepsilon_0}}{\log_{1/\sqrt{pq}} j }}\cdot\log_{1/\sqrt{pq}} j\big\}}.
\end{equation}
 As in (\ref{Theorem-st1}), we also have $ \lim_{j\to\infty}f_j=e$ and
\begin{equation*}\label{proof-thm3.3-7}
\lim_{j\to\infty}f_j^{{\ln{\lambda}}\cdot{\frac{\lfloor j^{1+\delta}\rfloor+1-2\lfloor j^{1+\delta}\rfloor pq}{\lfloor j^{1+\delta}\rfloor}
\cdot\frac{(1+\delta)\log_{1/\sqrt{pq}} j-\log_{1/\sqrt{pq}}{\varepsilon_0}}{\log_{1/\sqrt{pq}} j }}}=
\lambda ^{(1+\delta)(1-2pq)}<\lambda,
\end{equation*}
by choosing  $\delta>0$ small enough.  Then by \eqref{disconv1} we get
\begin{equation}\label{con-of-An}
\sum_{n=1}^{\infty}P(A_n)=\infty.
\end{equation}
Recall that $\{A_j,~j\geq1\}$ are defined in (\ref{A_j}). For $i<j$, we define
\begin{eqnarray*}
&&B_{i,j}=
\left\{
\begin{array}{cl}
\{M_{N_i}<\alpha_2(N_j)\},&\ \mbox{if}\ N_i\geq \alpha_2(N_j),\\
\Omega,&\ \mbox{otherwise},
\end{array}
\right.\\
&&C_{i,j}=\left\{M_{N_j-N_i}^{(N_i+1)}<\alpha_2(N_j)\right\}.
\end{eqnarray*}
We claim that
\begin{eqnarray}\label{proof-thm3.3-8}
P(A_j)=P(B_{i,j})P(C_{i,j})(1-o(1))\ \ \mbox{as}\ i<j\to\infty.
\end{eqnarray}
In fact, by the definitions of $A_j, B_{i,j}$ and $C_{i,j}$,  we know that the events $B_{ij}$ and $C_{ij}$ are independent and
\[
A_j=B_{i,j}\cap C_{i,j}\cap \left\{M^{(\{N_i-\alpha_2(N_j)\}^++1)}_{2\alpha_2(N_j)}<\alpha_2(N_j)\right\}
\]
with $x^+=x\vee 0$.
By Lemma \ref{LEM-S2N} we have
\begin{eqnarray}\label{proof-thm3.3-8-1}
P\left\{M^{(\{N_i-\alpha_2(N_j)\}^++1)}_{2\alpha_2(N_j)}<\alpha_2(N_j)\right\}&=&
P\left\{M_{2\alpha_2(N_j)}<\alpha_2(N_j)\right\} \nonumber \\
&\geq&P\left\{M_{2\alpha_2(N_j)}<\alpha_2(N_j)-1\right\}  \nonumber  \\
&=&1-o(1)    \ \mbox{as}\ j\to\infty.
\end{eqnarray}
Then \eqref{proof-thm3.3-8} holds immediately by \eqref{proof-thm3.3-8-1} and the independence of events $B_{ij}$ and $C_{ij}$.

In a similar way  we also have
\begin{eqnarray}\label{proof-thm3.3-11}
P(A_iA_j)=P(A_i)P(C_{i,j})(1-o(1))\ \mbox{as}\ i<j\to\infty.
\end{eqnarray}
For $N_i\geq \alpha_2(N_j)$ and $i<j$, by Theorem \ref{THM-EST-SN1}  for sufficiently large $j$
\[
P(B_{i,j})\geq P\left(M_{N_{j-1}}<\alpha_2(N_{j})\right)
\geq \left(1-W_{j}\right)^
{\big\lfloor\frac{N_{j-1}}{\alpha_2(N_j)+1}\big\rfloor-1}
\]
with $W_j:=(1-2pq)(pq)^{\alpha_2(N_j)/2}\alpha_2(N_j)\to0$ and $(1-W_j)^{1/{W_j}}\to e^{-1}$ as $j\to \infty$. Let $\lambda=1/\sqrt{pq}$.
The definition of $\alpha_2(N_j)$, the equality $\alpha_2(N_j)=\lfloor j^{1+\delta}\rfloor -1$, and \eqref{proof-thm3.3-5} together imply
\[
N_{j-1}\leq \log_{\lambda} \log_{\lambda} {N_{j-1}}\cdot \lambda^{\lfloor(j-1)^{1+\delta}\rfloor}\leq \log_{\lambda}{\{(j-1)^{1+\delta} /\varepsilon_0\}} \cdot\lambda^{\lfloor(j-1)^{1+\delta}\rfloor}
\]
for any $j$. Hence
\begin{eqnarray}\label{proof-thm3.3-9-c}
P(B_{i,j})&\geq &e^{-\lim\limits_{j\to\infty}W_j (\lfloor\frac{N_{j-1}}{\alpha_2(N_j)+1}\rfloor-1)}
= e^{-(1-2pq) \lim\limits_{j\to\infty} \frac{{N_{j-1}}}{\lambda^{\alpha_2(N_j)}}   }   \nonumber\\
&\geq &e^{-(1-2pq) \lim\limits_{j\to\infty}   \frac  { \log_{\lambda}{\{(j-1)^{1+\delta} /\varepsilon_0\}} }
{\lambda^{\{\lfloor j^{1+\delta}\rfloor-\lfloor(j-1)^{1+\delta}\rfloor+1\}} }   }      \nonumber\\
&\to&1\ \mbox{as}\ j\to\infty.
\end{eqnarray}
According to \eqref{proof-thm3.3-8}, \eqref{proof-thm3.3-11} and \eqref{proof-thm3.3-9-c} we easily have
\begin{eqnarray}\label{proof-thm3.3-12}
P(A_iA_j)={P(A_i)P(A_j)}(1+o(1))\ \mbox{as}\ i<j\to\infty.
\end{eqnarray}
Lemma \ref{B_C} holds by \eqref{con-of-An} and \eqref{proof-thm3.3-12}. We complete the proof.
\end{proof}

\begin{rem}  \label{sharp}
(i)~We conclude that the bounds given in \eqref{alpha2} and \eqref{alpha02} are sharp using our methods.  On the one hand, suppose that there exists  constant $a\geq0$ satisfying that  for almost all $\omega\in \Omega$, there exists a finite $N_0=N_0(\omega,\varepsilon)$ such that
\[
M_N\geq \alpha_1(N)+a:={\beta_1(N)} \text{~~for all $N\geq N_0$}.
\]
Similar to  \eqref{reason01}, we  use the inequality $\lfloor  x \rfloor\leq x$ with $x={\beta_1(N_j)}$  satisfying ${\beta_1(N_j)}=j-1$ to obtain
\[
 \sum_{j=1}^{\infty}P\left(M_{N_j}<\beta_1(N_j)\right)
\leq  C\sum_{j=1}^{\infty} {e_j}^{- \left\{ \lambda^{\epsilon-a}  \ln \lambda \cdot \log_{1/\sqrt{pq}} j\right\} },
\]
where $\lim\limits_{j\to\infty}e_j^{  \lambda^{\epsilon-a}  \ln \lambda}  =\lambda^{\lambda^{\varepsilon-a}} $.
By \eqref{conv1}, $\sum_{j=1}^{\infty}P\left(M_{N_j}<\beta_1(N_j)\right)<\infty$ if and only if $\lambda^{\varepsilon-a}>1$. Hence $a=0$ by  the arbitrariness of positive $\varepsilon$. Remark \ref{cong-wor} also tells us that why we cannot judge whether \eqref{alpha2}  holds when $\varepsilon=0$.

On the other hand, suppose that there exists  constant $b\in[0,2+\log_{1/\sqrt{pq}}2]$ satisfying that for almost all $\omega\in \Omega$, there exists an infinite sequence $N_j=N_j(\omega,\varepsilon) (j=1,2,...)$ of integers such that
\[
M_{N_j}<\alpha_2(N_j)-b:=\beta_2(N_j).
\]
Similar  to  \eqref{proof-thm3.3-6}, we  use the inequality $x <\lfloor  x \rfloor+1$ with $x={\beta_2(N_j)}$  satisfying ${\beta_2(N_j)}=\lfloor  j^{1+\delta}\rfloor-1$ to obtain
\begin{equation*}
\sum_{j=1}^{\infty}P(A_j)
\geq C\sum_{j=1}^{\infty}f_j^{-\big\{(1+\delta)(1-2pq)\lambda^b\ln\lambda\cdot\log_{1/\sqrt{pq}} j\big\}},
\end{equation*}
where $\lim_{j\to\infty}f_j^{ (1+\delta)(1-2pq)\lambda^b\ln\lambda}=\lambda ^{(1+\delta)(1-2pq)(pq)^{-b/2}}$.
Once we choose  $\delta$ small enough, by \eqref{disconv1} $\sum_{j=1}^{\infty}P(A_j)=\infty$ if and only if $(1-2pq)(pq)^{-b/2}<1$. However, for any $b>0$ we always have $(1-2pq)(pq)^{-b/2}\gg1$ when $pq$ closes to 0. Hence $b=0$. The rest is the same  as in  the proof of Theorem \ref{THM02} and we omit.

(ii)~We aim to give functions to bound $M_N$ for any $p\in(0,1)$. Actually, the lower and upper bounds can be further modified for particular $p$. See e.g. Hao et al. (2021) with $p=1/2$.

\end{rem}

\begin{proof}[{\bf Proof of Theorem 2.3*}]

Let $A_n=\{S^{(n-\gamma_n)}_{\gamma_n}= \gamma_n-1\}$. Similar to \eqref{F1p} we have
$$
       P(A_n)=P(M_{\gamma_n}= \gamma_n-1)\sim   (pq)^{\frac{\gamma_n}{2}}  .
$$
When $\sum_{n=1}^{\infty}{(pq)}^{\frac{\gamma_n}{2}}=\infty$, we have
$$
\sum_{n=1}^{\infty}P(A_n)=\infty.
$$
Following the method in the proof of Theorem \ref{THM02}, we have
 \begin{eqnarray*}
\frac{P(A_iA_j)}{P(A_i)P(A_j)}= 1+o(1)\ \mbox{as}\ j\to\infty.
\end{eqnarray*}
 Then ($\ref{condition2}$) holds and by Lemma $\ref{B_C}$ we obtain Theorem 2.3* (i).

And when $\sum_{n=1}^{\infty}{(pq)}^{\frac{\gamma_n}{2}}<\infty$, we have
$$
\sum_{n=1}^{\infty}P(A_n)<\infty.
$$
By the Borel-Cantelli lemma, we obtain Theorem 2.3* (ii).
\end{proof}

\subsection{Proof of Proposition \ref{THM-MN}}

Theorem \ref{THM01} implies
\begin{equation*}
\liminf_{N\to {\infty}}\frac{M_N}{\log_{{1/\sqrt{pq}}} N}\geq 1 \quad a.s.
\end{equation*}
We only need to prove
\begin{equation}\label{LEF}
\limsup_{N\to {\infty}}\frac{M_N}{\log_{1/\sqrt{pq}} N}\leq 1  \quad a.s.
\end{equation}
For any $\varepsilon>0$ and $N\in\mathbb N$, we introduce the following notations:
\begin{equation*}\label{NOTA-U-AN}
\begin{aligned}
&u:=\left\lfloor(1+\varepsilon)\log_{1/\sqrt{pq}} N\right\rfloor+1,\\
&A_N=\bigcup\limits_{k=1}^{N-u+1} \left\{S_{u}^{(k)}=u-1\right\}. \\
\end{aligned}
\end{equation*}
Thus we have $P(A_N)\leq 2N(pq)^{\lfloor\frac{u}{2}\rfloor}$, and by the definition of $u$ it is easy to have
\[
  P(A_N)\leq  2N(pq)^{-1} (pq)^{\frac{1+\varepsilon}{2}\log_{1/\sqrt{pq}} N}= 2(pq)^{-1} N^{-\varepsilon}.
\]
Let $T\in\mathbb N$ with $T\varepsilon>1$,  it holds that
\begin{equation*}\begin{aligned}
\sum_{k=1}^{\infty}P(A_{k^T})\leq2(pq)^{-1} \sum_{k=1}^{\infty}\frac{1}{k^{T\varepsilon}}<{\infty},
\end{aligned}\end{equation*}
which, again by  the Borel-Cantelli lemma
\begin{equation*}
P(A_{k^{T}} \ i.o.)=P\Big(\limsup_{k\to{\infty}} \bigcup\limits_{i=1}^{k^T-{\tilde{u}}+1}\left\{S_{\tilde{u}}^{(i)}={\tilde{u}}-1\right\}\Big)=0
\end{equation*}
with ${\tilde{u}}:=\left\lfloor(1+\varepsilon)\log_{1/\sqrt{pq}} k^T\right\rfloor+1$. It follows by the definition of $\tilde{u}$ that
\begin{eqnarray*}\label{4.3}
\limsup_{k\to\infty}\frac{M_{k^T}}{\log_{1/\sqrt{pq}} k^T}\leq 1+\varepsilon \quad\quad a.s.
\end{eqnarray*}
So for any $N\in\mathbb N$, we choose $k\in\mathbb N$ such that $k^T\leq N\leq (k+1)^T$ and obtain
\[
M_N\leq M_{(k+1)^T}\leq (1+\varepsilon)\log_{1/\sqrt{pq}} (k+1)^T\leq (1+2\varepsilon)\log_{1/\sqrt{pq}} k^T\leq (1+2\varepsilon)\log_{1/\sqrt{pq}} N
\]
with probability 1 for all but finitely many $N$. Hence we get that $\limsup\limits_{N\to {\infty}}\frac{M_N}{\log_{1/\sqrt{pq}} N}\leq 1+2\varepsilon  \ a.s.$ \eqref{LEF} holds by the arbitrariness of $\varepsilon$. \hfill\fbox

\bigskip

{ \noindent {\bf\large Acknowledgments} \quad We thank the two referees for helpful comments and suggestions, which helped to improve the presentation of this note. We also thank Prof. Ze-Chun Hu for helpful discussion. This work was supported by the National National Science Foundation of China (Nos. 12101429, 12171335), the Science Development Project of Sichuan University (2020SCUNL201) and the Fundamental Research Funds for the central Universities of China.

\end{document}